\documentclass[12pt,reqno]{amsart}
\usepackage{amsmath,amssymb}
\usepackage{hyperref}
\usepackage[english]{babel}
\usepackage{caption}
\usepackage{graphicx}
\usepackage{float}
\usepackage{pgf,tikz}
\numberwithin{equation}{section}
\tikzstyle{vertex}=[circle,draw, inner sep=0pt, minimum size=1.5pt]

\newtheorem{thm}{Theorem}[section]

\newtheorem{lem}[thm]{Lemma}

\newtheorem{cor}[thm]{Corollary}

\theoremstyle{definition}
\newtheorem{definition}[thm]{Definition}
\newtheorem{example}[thm]{Example}
\newtheorem{conjecture}[thm]{Conjecture}

\newtheorem{remark}[thm]{Remark}
\newtheorem{claim}[thm]{Claim}

\newtheorem{obs}[thm]{Observation}

\newcommand{\C}{\mathcal{C}}
\newcommand \reg{\operatorname{reg}}

\newcommand{\pol}{\operatorname{pol}}
\newcommand{\even}{\operatorname{even}}

\begin{document}
 \title{Regularity of Powers of edge ideals of Unicyclic Graphs }

 \author[Ali Alilooee]{Ali Alilooee}
\address{University of Wisconsin-Stout, Department of Mathematics and Statistics,
Jarvis Hall-Science Wing,
Menomonie, WI, USA }
\email{a20480m2018@gmail.com}
%\urladdr{http://www.aliloee.ir/}

\author{Selvi Kara}
\address{University of South Alabama, Department of Mathematics and Statistics, 411 University Boulevard North, Mobile, AL 36688-0002, USA}
\email{selvi@southalabama.edu}

\author{S. Selvaraja}
\address{ Institute of Mathematical Sciences, C. I. T. Campus, Chennai 600 113, INDIA}
\email{selva.y2s@gmail.com}


\thanks{AMS Classification 2010: 05C25, 05C38, 05E40, 13D02, 13F20}

%\thanks{The third author is supported by the National Board for Higher Mathematics, India.}

\maketitle
 \begin{abstract}
Let $G$ be a finite simple graph and $I(G)$ denote the corresponding
edge ideal. In this paper we prove that if $G$ is a unicyclic graph then
for all $s \geq 1$ the regularity of $I(G)^s$ is exactly $2s+\reg(I(G))-2$.
We also give a combinatorial characterization of unicyclic graphs with regularity
$\nu(G)+1$ and $\nu(G)+2$ where
$\nu(G)$ denotes the induced matching number of $G$.
 \end{abstract}

 \section{Introduction}
 
 Let $G = (V(G) , E(G))$ denote a finite simple (no loops, no multiple edges) undirected graph with vertices
$V(G) = \{x_1,\ldots,x_n\}$ and edge set $E(G)$. By identifying the vertices with the variables in the polynomial ring
$R=K[x_1,\ldots, x_n]$ where $K$ is a field, we can associate
each graph $G$ to a monomial ideal $I(G)$ generated by the set
$\{x_i x_j \mid \{x_i , x_j \}  \in E (G) \}$.
The ideal $I(G)$ is called the \textit{edge ideal} of $G$. Recently, building a dictionary between combinatorial data of graphs and the algebraic
properties of the corresponding edge ideals has been studied by various authors, (cf. \cite{BHT}, \cite{BC2}, \cite{Frob}, \cite{Ha2}, \cite{HVanT},
\cite{Jacques}, \cite{JNS}, \cite{Katzman},
\cite{MSY}, \cite{Morey}, \cite{Wood}, \cite{Zheng}).
In particular, establishing a relationship between Castelnuovo-Mumford regularity of the
edge ideals and
combinatorial invariants associated with graphs such as induced matching number, matching number and co-chordal cover number is an
active research topic, (cf.
\cite{HVanT}, \cite{Katzman}, \cite{Wood}).

Our motivation to study regularity of powers of edge ideals springs from a famous result: for a homogeneous ideal $I$ in a polynomial ring,
$\reg (I^s)$ is asymptotically a linear function for
$s \gg 0,$ (cf.  \cite{Chardin}, \cite{CHT}, \cite{Kodi}, \cite{TW}),
i.e., there exist non-negative integers $a$, $b$, $s_0$ such that
$$\reg(I^s)=as+b \text{ for all $s \geq s_0$}.$$
While the coefficient $a$ is well-understood (\cite{CHT}, \cite{Kodi}, \cite{TW}), the constants $b$ and $s_0$ are
quite mysterious.
In this regard, there has been an interest in finding the exact form of the linear function and determining the stabilization index $s_0$
where $\reg(I^s)$ becomes linear (cf. \cite{Berle}, \cite{Chardin2}, \cite{EH}, \cite{EU}, \cite{Ha}).
It turns out that even in the case of monomial ideals it is challenging to find the linear function and $s_0$ (cf. \cite{Conca}, \cite{Ha2}).
In this paper, we consider $I = I(G)$, the edge ideal of $G$. In this case, there exist integers $b$ and $s_0$ such that
$\reg(I^s) = 2s + b$ for all $s\geq s_0$. 
Our objective in this paper is to find $b$ and
$s_0$ in terms of combinatorial invariants of the graph $G$ when $G$ is a \textit{unicyclic graph}, i.e. a graph containing exactly one cycle. 
There are few classes of graphs for which $b$ and $s_0$ are explicitly computed
(see, for example, \cite{AB}, \cite{Banerjee}, \cite{BHT}, \cite{Frob}, 
 \cite{JNS}, \cite{MSY}).

In \cite{BHT}, Kara, H\`{a} and Trung
proved that $2s+\nu(G)-1 \leq \reg(I(G))^s$ for $s\geq 1 $ and any graph $G$ where $\nu(G)$ denote the induced matching number of
$G$.
They also proved that the equality holds for
$s \geq 1$ when $G$ is a forest and for  $s \geq 2$ when $G$ is a cycle. 
A natural class of graphs to
consider next is unicyclic graphs. The regularity of edge ideal of a unicyclic graph is investigated in 
\cite{BC2} and the depth of powers of edge ideal of a unicyclic graph have been studied in \cite{trung}.
Throughout the paper, we  shall restrict our attention to unicyclic graphs which are connected and
not cycles.

We then compute the regularity of powers of edge ideals of unicyclic graphs. The main result of the paper is the following.

\begin{thm}\textup{(Theorem \ref{main}.)}\label{second}
 If $G$ is a unicyclic graph, then for all $s \geq 1$,
$$\reg(I(G)^s)=2s+\reg(I(G))-2.$$
\end{thm}

Note that for this class of graphs, we have $b=\reg(I(G))-2$ and $s_0=1$. As an immediate consequence,
we derive one of the main results of \cite{MSY}, that the above equality holds for whiskered cycle graphs.

To prove Theorem \ref{second}, we establish the upper bound $\reg (I(G)^s) \leq 2s+\reg(G)-2$ for all $s\geq1$
when $G$ is a unicyclic graph (Lemma \ref{big}). This upper bound coupled with the lower bound given in \cite[Theorem 4.5]{BHT} leads us to the following.
\begin{equation*}
2s+\nu(G)-1 \leq \reg(I(G)^s) \leq 2s+\reg(I(G))-2.
\end{equation*}
It follows from the above inequalities that $\reg (I(G)^s)= 2s+\nu(G)-1$ for all $s \geq 1$ when $\reg(I(G))= \nu(G)+1.$
In the case where  $\reg(I(G))=\nu(G)+2,$ we present an induced subgraph of $G,$ say $H,$ such that $\reg(I(H)^s)= 2s+\nu(G).$
Thus by making use of \cite[Corollary 4.3]{BHT} and the upper bound, we prove that $\reg (I(G)^s)=2s+\nu(G)$ for all $s \geq 1.$

The first key step in the proof of the main result is to compute $\reg (I(G))$ for a unicyclic graph $G$ and the results obtained in this step are of independent interest. It is known that for any unicyclic graph $G$,
\begin{equation*}
\nu(G) + 1 \leq \reg(I(G)) \leq \nu(G)+2.
\end{equation*}
The lower bound was proved by Katzman, \cite{Katzman} and the upper bound
was proved by B{\i}y{\i}ko{\u{g}}lu and
Civan, \cite{BC2}. In this paper, we provide the complete combinatorial characterization of unicyclic graphs where the regularity is  $\nu(G)+1$ and $\nu(G)+2.$ 

In the pursuit of the desired characterization, we make use of an important yet a basic observation related to the structure of unicyclic graphs: a unicyclic graph is obtained from a cycle by attaching trees to some of the vertices of the cycle. We then call those vertices of the cycle as roots and introduce the notation $\Gamma(G)$ to denote the neighbors of roots which are not on the cycle.

Our first result in this context gives the characterization of unicyclic graphs when $\reg(I(G))=\nu(G)+2.$  

\begin{thm}\textup{(Corollary \ref{unireg1}.)} \label{first}
Let $G$ be a unicyclic graph with cycle $C_n.$ Then $\reg(I(G))=\nu(G)+2$ if and only if $n \equiv 2 ~( mod~3)$ and $\nu(G \setminus \Gamma(G))=\nu(G)$ where $G \setminus \Gamma(G)$ is the induced subgraph of $G$ on $V(G) \setminus \Gamma(G).$
\end{thm}

In order to prove Theorem \ref{first}, we provide necessary conditions for a unicyclic graph to have regularity
$\nu(G)+1$ (Lemma \ref{unireg_1} and Theorem \ref{unireg_2}) and $\nu(G)+2$ (Theorem \ref{regupper}).
The characterization of unicyclic graphs with regularity $\nu(G)+1$ (Corollary \ref{unireg2}) follows from Theorem \ref{first}.

Our paper is organized as follows. In section \ref{pre}, we collect the necessary notation and terminology that will be used in the paper.
In Section \ref{uni_reg}, we prove Theorem \ref{first}.  Section \ref{reg_path_cycle} is
devoted to finding bounds for regularity of special colon ideals related to paths and cycles. Finally, we prove Theorem \ref{second} in Section \ref{reg_power} by using the main result of Section \ref{reg_path_cycle}.

\section{Preliminaries}\label{pre}

In this section, we set up the basic definitions and terminology needed for the main results.

Let $G$ be a finite simple graph with vertex set $V(G)$
and edge set $E(G)$. For a vertex $x$ in a graph $G$, let $N_G(x)=\{y \in V(G) \mid \{x,y\} \in E(G)\}$
be the set of neighbors of $x$ and set
$N_G[x]=N_G(x) \cup \{x\}.$ An edge $e$ is \textit{incident} to a vertex $x$ if $x \in e.$ If $e=\{x,y\}$ then set $N_G[e]= N_G[x]\cup N_G[y].$ We often use $xy \in E$ instead of $\{x,y\} \in E(G).$ By abusing notation, we use the notation $xy$ to refer to both the edge $xy \in E(G)$ and the monomial $xy \in I(G).$

The \textit{degree} of a vertex $x \in V(G),$ denoted by $\deg_G(x),$ is the number of edges incident to $x.$ If $\deg_G (x)=1,$
then $x$ is called a \textit{leaf} of $G.$ If $x$ is a leaf and $N_G(x)=\{y\},$
then we also call the edge $e= \{x,y\}$ a \textit{leaf (also called whisker)} of $G.$
Let $C_n$ denote the cycle on $n$ vertices and $P_n$ denote the path on $n$ vertices. The length of a path, or a cycle is its number of edges.  

Let $e \in E(G)$, then define $G \setminus e$ to be the subgraph of $G$ obtained from $G$ by deleting the edge $e$ but keeping
its vertices. If $W \subseteq V(G)$ in $G,$ then $G \setminus W$ denotes the subgraph of $G$ with the vertices in $W$ and all incident
edges deleted. When $W= \{x\}$ consists of a single vertex, we shall write $G \setminus x$ instead of $G \setminus \{x\}.$

A graph $H$ is called an \textit{induced subgraph} of $G$ if the vertices of $H$ are the vertices of $G,$ and for the vertices $x$ and $y$ in $H,$
$\{x,y\}$ is an edge in $H$ if and only if $\{x,y\}$ is an edge in $G.$ The induced subgraph of $G$ over a subset $W \subseteq V(G)$
is obtained by deleting all the vertices that are not in $W$ from $G.$

 Let $G$ and $H$ be graphs. Their union, denoted by $G \cup H$,  is a graph with the vertex set
 $V(G)\cup V(H)$ and edge set  $E(G)\cup E(H).$ If $G$ and $H$ disjoint graphs (i.e., $V(G) \cap V(H) =\emptyset$),
 we denote the disjoint union of $G$ and $H$ by $G \coprod H.$

 A \textit{matching} in a graph $G$ is a collection of pairwise disjoint edges $\{e_1, \ldots, e_s\}$.
We call a collection of edges $\{e_1, \ldots, e_s\}$ an \textit{induced matching}
if they form a matching in $G,$ and they are exactly the edges of the induced subgraph of $G$ over the vertices
$\bigcup_{i=1}^n e_i.$ The largest size of an induced matching in $G$ is called its \textit{induced matching number} and denoted by $\nu(G).$
Note that if $H$ is an induced subgraph of $G,$ then $\nu(H) \leq \nu(G).$ Furthermore, if $G$ and $H$ are disjoint graphs,
then $\nu(G \coprod H) =\nu(G)+\nu(H).$

\begin{example}\label{Example}
Let G be a graph with $V(G)=\{x_1,\ldots,x_7\}$.

\captionsetup[figure]{}
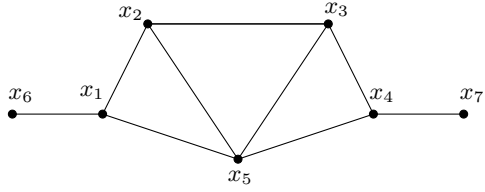
\begin{figure}[H]

\begin{tikzpicture}[scale=.6]
%\clip(-0.34,-5.64) rectangle (21.78,6.56);
\draw (6.,6.)-- (10.,6.);
\draw (6.,6.)-- (5.,4.);
\draw (3.,4.)-- (5.,4.);
\draw (5.,4.)-- (8.,3.);
\draw (10.,6.)-- (11.,4.);
\draw (11.,4.)-- (13.,4.);
\draw (11.,4.)-- (8.,3.);
\draw (10.,6.)-- (8.,3.);
\draw (6.,6.)-- (8.,3.);
\draw (6.,6.)-- (10.,6.);
\begin{scriptsize}
\draw [fill=black] (6.,6.) circle (2.5pt);
\draw[color=black] (5.63,6.24) node {$x_2$};
\draw [fill=black] (10.,6.) circle (2.5pt);
\draw[color=black] (10.19,6.32) node {$x_3$};
\draw [fill=black] (5.,4.) circle (2.5pt);
\draw[color=black] (4.77,4.48) node {$x_1$};
\draw [fill=black] (3.,4.) circle (2.5pt);
\draw[color=black] (3.19,4.42) node {$x_6$};
\draw [fill=black] (8.,3.) circle (2.5pt);
\draw[color=black] (8.05,2.6) node {$x_5$};
\draw [fill=black] (11.,4.) circle (2.5pt);
\draw[color=black] (11.19,4.42) node {$x_4$};
\draw [fill=black] (13.,4.) circle (2.5pt);
\draw[color=black] (13.19,4.42) node {$x_7$};
\end{scriptsize}
\end{tikzpicture}
\caption{A finite simple graph}

\end{figure}
\noindent
Then $\{x_1x_6,x_2x_3,x_4x_7\}$ forms a matching,
but not an induced matching (the induced subgraph on $\{x_1,x_2,x_3,x_4,x_6,x_7\}$ also contains edges
$\{x_1x_2,x_3x_4\}$). The induced
matching number $\nu(G)$ is 2.

\end{example}

The following observation will be used repeatedly in our proofs.
\begin{obs}\label{indmatch} Let $G$ be a graph with a leaf $u$ and its unique neighbor $v,$ say $e=\{u,v\}.$
If $\{e_1, \ldots, e_s\}$ is an induced matching in $G \setminus N_G [v],$ then $\{e_1, \ldots, e_s,e\} $ is an induced matching in $G.$
Therefore, $\nu(G\setminus N_G [v]) +1 \leq \nu(G).$
\end{obs}

\begin{definition} Let $R$ be a standard graded polynomial ring over a field $K.$ 
The \textit{Castelnuovo-Mumford regularity}
(or  regularity) of a finitely generated graded $R$ module $M,$ written $\reg(M)$ is given by
$$\reg (M) := \max \{j-i \mid \textrm{Tor}_i(M,K)_j \neq 0\} .$$

\end{definition}

When discussing the regularity of edge ideals, for simplicity of notation, we shall use $\reg(G)$ to also refer to $\reg(I(G))$.

Let $I$ be a non-zero proper homogeneous ideal of $R$. Then it is straight from the definition that $\displaystyle\reg\left(R/I\right)=\reg(I)-1$.

For a homogeneous ideal $I$ in $R$ and any homogeneous element $M \in R$ of degree $d,$ the following short exact sequence is a standard tool in commutative algebra:

\begin{eqnarray}\label{ses}
0 \longrightarrow \frac{R}{I:M} (-d) \xrightarrow{\cdot M} \frac{R}{I} \longrightarrow \frac{R}{I+M} \longrightarrow 0
\end{eqnarray}

By taking the long exact sequence of local cohomology modules associated to \ref{ses}, we have the following useful inequality

\begin{eqnarray}\label{sesreg}
 \reg (I) \leq \max \{ \reg (I:M)+d, \reg (I,M)\}.
\end{eqnarray}

We use the following well-known theorem to prove an upper bound for the regularity of edge ideals inductively:

 \begin{thm}\label{Inequalities}
Let $G=(V(G),E(G))$ be a graph.
\begin{enumerate}
\item \cite[Lemma 3.1]{Ha2} If $H$ is an induced subgraph of $G,$ then $\reg(I(H)) \leq \reg(I(G)).$
\item \cite[Lemma 2.10]{DHS}  Let $x \in V(G).$ Then $$\reg(I(G)) \leq \max\{ \reg(I(G\setminus x)), \reg(I(G\setminus N[x]))+1 \}.$$
%\item Let $e \in E(G).$ Then $$\reg (I(G)) \leq \max\{ \reg(I(G \setminus e)), \reg(I(G_e))+1 \}.$$
\end{enumerate}
\end{thm}

The concept of even-connectedness was introduced by Banerjee in \cite{Banerjee}. This notion has emerged as a fine tool in the inductive process of computing asymptotic 
regularity.
\begin{definition}\label{even_connected} Let $G=(V(G),E(G))$ be a graph. Two vertices $u$ and $v$ ($u$ may be the same as $v$) are
said to be even-connected with respect to an $s$-fold product $e_1\cdots e_s$ where $e_i$'s are edges of $G$, not necessarily distinct,
if there is a path $p_0p_1\cdots p_{2k+1}$, $k\geq 1$ in $G$ such that:
\begin{enumerate}
 \item $p_0=u,p_{2k+1}=v.$
 \item For all $0 \leq l \leq k-1,$ $p_{2l+1}p_{2l+2}=e_i$ for some $i$.
 \item For all $i$, $ \mid\{l \geq 0 \mid p_{2l+1}p_{2l+2}=e_i \}\mid
   ~ \leq  ~ \mid \{j \mid e_j=e_i\} \mid$.
 \item For all $0 \leq r \leq 2k$, $p_rp_{r+1}$ is an edge in $G$.
\end{enumerate}
\end{definition}
\begin{example}
In Example \ref{Example} if we set $e_1=x_1x_5$ and $e_2=x_3x_4$, then we have $x_6$ and $x_7$ are even-connected in $G$ with respect to
$e_1e_2$ since we have the path
$(p_0=x_6)x_1x_5x_3x_4(x_7=p_5)$. Also note that $x_2$ is even-connected to itself with respect to $e_1e_2$ since we have the path $x_2x_1x_5x_4x_3x_2$.
\end{example}

As \ref{sesreg} points out, analyzing the ideal $(I(G)^{s+1}:M)$ for a minimal monomial generator of $I(G)^s$ can be used as an important asset in the computation of asymptotic regularity. In \cite{Banerjee}, it is proved that these ideals are generated in degree two for any graph and the description of the generators of this ideal is given by using the notion of even-connection.

\begin{thm}\label{even_connec_equivalent}\cite[Theorem 6.1 and Theorem 6.7]{Banerjee} Let $G$ be a graph with edge ideal
$I = I(G)$, and let $s \geq 1$ be an integer. Let $M$ be a minimal generator of $I^s$.
Then $(I^{s+1} : M)$ is minimally generated by monomials of degree 2, and $uv$ ($u$ and $v$ may
be the same) is a minimal generator of $(I^{s+1} : M )$ if and only if either $\{u, v\} \in E(G) $ or $u$ and $v$ are even-connected with respect to $M$.
 \end{thm}

Polarization is a process to obtain a squarefree monomial ideal from a given monomial ideal and it behaves well under regularity.
For details of polarization we refer to \cite{Faridi} and \cite{Herzog'sBook}.
\begin{definition}\label{pol_def}
Let $M=x_1^{a_1}\dots x_n^{a_n}$ be a monomial in  $R=k[x_1,\dots,x_n]$. Then we define the squarefree monomial $P(M)$ ({\it polarization} of $M$) as
$$P(M)=x_{11}\dots x_{1a_1}x_{21}\dots x_{2a_2}\dots x_{n1}\dots x_{na_n}$$ in the polynomial ring $S=k[x_{ij} \mid 1\leq i\leq n,1\leq j\leq a_i]$.
If $I=(M_1,\dots,M_q)$ is an ideal in $R$, then the polarization of $I$, denoted by $I^{\pol}$, is define as $I^{\pol}=(P(M_1),\dots,P(M_q))$.
\end{definition}

\begin{cor} \cite[Corollary 1.6.3.d]{Herzog'sBook} \label{betti} Let $I \subset R$ be a monomial ideal and $I^{\pol}\subset S$ be its polarization. Then $\reg (R/I) =\reg (S/I^{\pol}).$
\end{cor}

%
%\iffalse
%Let $G$ be a graph and let $I(G)$ denote the edge ideal of $G$. Then for
%any $s \geq 1$ and edges $e_1, \ldots, e_s$ of $G$, $I^{pol} =
%(I(G)^{s+1} : e_1\cdots e_s)^{pol}$ is a squarefree quadratic
%monomial ideal, by Theorem \ref{even_connec_equivalent}. Hence there
%exists a graph $G'$ associated to $I^{pol}$.
%\fi
%\begin{example}
%Let $G$ be the graph of the Example \ref{Example} and
%$$I(G)=(x_1x_6,x_1x_2,x_1x_5,x_2x_5,x_2x_3,x_3x_5,x_5x_4,x_4x_3,x_4x_7)\subset K[x_1,\dots,x_7].$$
%Then $I=(I(G)^{2}:x_2x_3)=I(G)+(x_5^2,x_1x_4)$. Therefore,
%$I^{pol}\subset k[x_1,\ldots,x_7,y_1]$ is given by
%$I^{pol}=I(G)+(x_5y_1,x_1x_4)$.
%\end{example}

\section{Regularity of unicyclic graphs} \label{uni_reg}

The regularity of unicyclic graphs is studied  in \cite{BC2} and the authors proved that it is either  $\nu(G)+1$  or  $\nu(G)+2.$ 
 In this section, we provide the combinatorial characterization of unicyclic graphs with regularity $\nu(G)+1$ and regularity $\nu(G)+2$.

 The following theorem by B{\i}y{\i}ko{\u{g}}lu and Civan turns out to be crucial in proving our main results.
   \begin{thm}\cite[Corollary 4.12]{BC2} \label{unicyclicreg} If $G$ is a unicyclic graph, then
 $$\nu(G)+1 \leq \reg(I(G)) \leq \nu(G)+2.$$
 \end{thm}

We start our investigation by computing the regularity of a unicyclic graph with exactly one leaf.

\begin{lem}\label{whiskerreg}
Let $G$ be obtained from $C_n:x_1x_2 \cdots x_n$ by attaching a leaf, say $y,$ to a vertex  $x_i$ where $1 \leq i \leq n$. Then
$$\reg(I(G))=\nu(G)+1.$$
\end{lem}

\begin{proof}
By \cite[Lemma 3.25]{biyikoglu_civan}, we have
\begin{eqnarray*}
\reg(I(G))=\reg(I(G \setminus y)) &\text{or}& \reg(I(G))=\reg(I(G \setminus N_G[x_i]))+1.
\end{eqnarray*}
If $\reg(I(G))=\reg(I(G \setminus y))$, then by \cite[Theorem 7.6.28]{Jacques}
$$\reg(I(G))=\reg(I(G \setminus y))=\reg(I(C_n)).$$
If $n \equiv \{0,1\}(mod~3)$, then $\reg(I(G))=\reg(I(C_n))=\nu(G)+1$. If $n \equiv 2 (mod~3)$, then $\reg(I(G))=\reg(I(C_n))=\nu(C_n)+2=\nu(G)+1$.
If $\reg(I(G))=\reg(I(G \setminus N_G[x_i]))+1$, then by \cite[Theorem 2.18]{Zheng} and Observation \ref{indmatch},
$$\reg(I(G))=\reg(I(G \setminus N_G[x_i]))+1=\nu(G \setminus N_G[x_i])+2 \leq \nu(G)+1.$$
By \cite[Lemma 2.2]{Katzman}, we have  $\reg(I(G))=\reg(I(G \setminus N_G[x_i]))+1=\nu(G)+1$. Therefore $\reg(I(G))=\nu(G)+1$.
\end{proof}

The first general case we consider is based on the size of the cycle in a unicyclic graph.

\begin{lem}\label{unireg_1} Let $G$ be a unicyclic graph with cycle $C_n$.
If $n \equiv \{0,1\} ~(mod~3)$, then $$\reg(I(G))=\nu(G)+1.$$
\end{lem}
\begin{proof}
 By \cite[Lemma 2.2]{Katzman}, we have $\reg(I(G)) \geq \nu(G)+1$.
 It suffices to show that $\reg(I(G)) \leq \nu(G)+1$. Let $F$ be the graph with
 $E(F)=E(G) \setminus E(C_n)=\{f_1,\ldots,f_k\}$.
We use induction on $k$. If $k=1$, then by Lemma \ref{whiskerreg}, we have $\reg(I(G)) = \nu(G)+1$.
Assume that $k \geq 2$. There is a leaf $y$ in $G$ such that $\{x\}=N_G(y)$. Set $G'=G\setminus x$ and
$G''=  G \setminus N_G[x]$. By Theorem \ref{Inequalities}, we have
\begin{equation}\label{equation1}
\reg(I(G)) \leq \max \{ \reg(I(G')), \reg(I(G''))+1\}.
\end{equation}
Note that $G'$ is a unicyclic graph, a forest, or a cycle and also $G''$ is a unicyclic graph, a forest, or a cycle.
Therefore
$$\reg(I(G'))=\nu(G')+1 \leq \nu(G)+1$$
by the induction hypothesis or \cite[Theorem 2.18]{Zheng} or \cite[Theorem 7.6.28]{Jacques}.

By Observation \ref{indmatch}, $\nu(G'')<\nu(G)$.
Then it follows from  the induction hypothesis or \cite[Theorem 2.18]{Zheng} or \cite[Theorem 7.6.28]{Jacques} that
$$\reg(I(G''))=\nu(G'')+1 \leq \nu(G).$$
Therefore $\reg(I(G)) \leq \nu(G)+1$ by Equation \ref{equation1}.
\end{proof}

A unicyclic graph can be viewed as a graph obtained by attaching trees to some vertices of a cycle $C_n.$ Those vertices of the cycle $C_n$ can be thought as the roots of 
attached trees.
\begin{figure}[h]
  \includegraphics[width=0.5\linewidth]{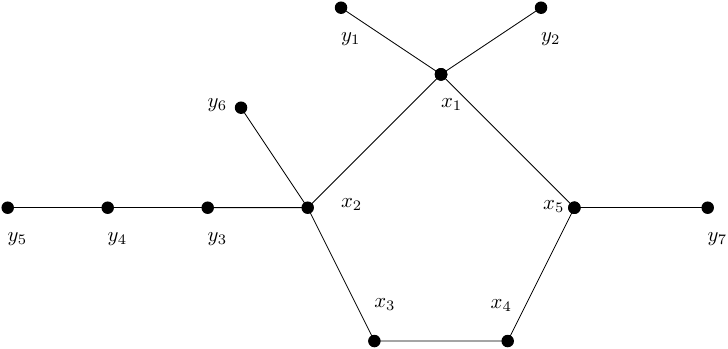}
  \caption{A unicyclic graph}
  \label{fig:runningex}
\end{figure}

In the above graph, vertices $x_1, x_2$ and $x_5$ can be considered as roots of the trees. 
Let $T_1$ be the tree with the root $x_1$ and the edges $\Big \{\{x_1,y_1\},\{x_1,y_2\} \Big\},$  
$T_2$ be the tree with root $x_2$ and the edges  $\Big\{\{x_2,y_6\},\{x_2,y_3\},\{y_3,y_4\},\{y_4,y_5\}\Big \},$ and $T_3$ be the tree with root $x_5$ and the edge  $\Big \{\{x_5,y_7\}\Big\}.$

Understanding the relationship between the induced matching numbers of a unicyclic graph and collection of some induced subgraphs of the attached rooted trees plays a key 
role in the classification of regularity of unicyclic graphs. For this purpose, we introduce the following notation and use it in the rest of the paper.

\subsection{Notation.} \label{notation} Let $G$ be a unicyclic graph with cycle $C_n:x_1x_2 \cdots x_n$ and $T_1,\ldots,T_m$ be the rooted trees of $G$ with roots $\{x_{i_1},\ldots,x_{i_m}\} \subseteq V(C_n)$. Consider all the neighbors of $\{x_{i_1},\ldots,x_{i_m}\} $ in the rooted trees and denote that collection by $\Gamma(G).$
$$\Gamma(G)= \bigcup_{j=1}^m N_{T_j}(x_{i_j}):=\{y_1,\ldots,y_t\} \subseteq \bigcup_{j=1}^m V(T_j).$$
Note that none of the vertices in $\Gamma(G)$  can be a vertex on the cycle $C_n.$
Let $H_j$ be the induced subgraph of $T_j$ obtained by deleting the elements of $\Gamma(G)$  that are vertices in $T_j.$
$$H_j=T_j \setminus \{z_k \mid z_k \in V(T_j) \cap \Gamma(G) \}.$$ 
Note that $H_j$ is either a forest or a tree, and $H_j$'s are disjoint. Thus
\begin{eqnarray}\label{indequal}
G \setminus \Gamma(G) =  C_n \coprod \Big(\coprod_{j=1}^{m}H_j\Big) \textrm{ and } \nu (G \setminus \Gamma(G) ) =  \nu(C_n)+\sum_{j=1}^m \nu(H_j).
\end{eqnarray}

\begin{example} Let $G$ be the graph in Figure \ref{fig:runningex}. Then  $\Gamma(G) =\{y_1,y_2,y_3,y_6,y_7\}$ and  $G \setminus \Gamma(G) = C_5 \coprod  H_2$ where $\{y_4,y_5\}$ is the only edge of $H_2.$  
\end{example}

It turns out that induced matching of $G$ is preserved under deletion of vertices of $\Gamma(G)$ if it is preserved  on each rooted tree.

\begin{lem}\label{ind_obs} If $\nu(H_j)=\nu(T_j)$ for all $1 \leq j \leq m$, then $\nu(G \setminus \Gamma(G))=\nu(G).$
\end{lem}

\begin{proof}
Since $G \setminus \Gamma(G)$ is an induced subgraph of $G,$ we have $\nu(G \setminus \Gamma(G))\leq \nu(G).$
It remains to prove the reverse inequality. It follows from the assumption and Equation \ref{indequal} that
$$
   \begin{array}{lcll}
\nu (G \setminus \Gamma(G))&=& \nu(C_n)+\sum_{j=1}^m \nu(T_j). 
    \end{array}
$$
If $\mathcal{C}$ is an induced matching of $G,$ then $\mathcal{C}$ can be decomposed as a union of an induced matching in $C_n$  and induced matchings in $T_j$'s. Hence  
$$\nu(G) \leq  \nu(C_n)+\sum_{j=1}^m \nu(T_j). $$ 
It concludes that $\nu(G \setminus \Gamma(G))=\nu(G).$
\end{proof}

With the help of Lemma \ref{ind_obs}, we get another sufficient condition for $\reg (I(G))=\nu(G)+1$ when $G$ is a unicyclic graph.

\begin{thm}\label{unireg_2} If $\nu(G \setminus \Gamma(G))<\nu(G)$, then $\reg(I(G))=\nu(G)+1.$
\end{thm}

\begin{proof} By \cite[Lemma 2.2]{Katzman}, we have $\reg(I(G)) \geq \nu(G)+1$.
It suffices to show that $\reg(I(G)) \leq \nu(G)+1$. Since $\nu(G \setminus \Gamma(G))<\nu(G),$ by Lemma \ref{ind_obs}, there exists a rooted tree
 $T_r$ with root $x_{i_r}$ such that $\nu(H_r) <\nu(T_r)$  for some $r \in \{1, \ldots, m\}$.

Let $G_1=G\setminus x_{i_r}$ and $G_2=  G \setminus N_G[x_{i_r}]$. By Theorem \ref{Inequalities}, we have
\begin{equation}\label{eq}
\reg(I(G)) \leq \max \{ \reg(I(G_1)),~\reg(I(G_2))+1\}.
\end{equation}

Since $G_1$  and $G_2$ are forests, by \cite[Theorem 2.18]{Zheng},  we have 
$$\reg(I(G_i))=\nu(G_i)+ 1 \textrm{ for } i=1,2.$$

It is clear that $\reg(I(G_1)) \leq \nu(G)+1.$ Thus proving $\nu(G_2)+1 \leq \nu(G)$ yields to the desired equality.

Observe that $G_2$ can be written as a disjoint union  of $H$ and $H_r$ where $H$ is the induced subgraph of $G$ obtained by deleting the vertices of $T_r$ and 
$N_G[x_{i_r}].$  Let $\C$ be an induced matching of $G_2.$ Then $\C$ can be decomposed as a disjoint union of an induced matching in $H$ and $H_r.$ Let $\C_H$ and 
$\C_{H_r}$ denote the corresponding induced matchings of $H$ and $H_r,$ respectively.

Suppose $\C_{H_r}= \{h_1,\ldots,h_\beta\}.$ Since $\nu(H_r) <\nu(T_r),$ there exists an edge $e$ incident to $z_j$ for some $z_j \in V(T_r) \cap \Gamma(G)$ such that $\{e,h_1',\ldots,h_\beta'\}$ is an induced matching in $T_r.$ Note that $\{h_1',\ldots,h_\beta'\}$ is an induced matching in $H_r.$ Furthermore, $\C_{H} \cup \{e,h_1',\ldots,h_\beta'\}$ is an induced matching in $G$ due to the fact that neighbors of $x_{i_r}$ are deleted to construct $H$ and $z_j \in N_G[x_{i_r}].$ Therefore $\reg(I(G_2))=\nu(G_2)+1 \leq \nu(G).$
It follows from Equation \ref{eq} that $\reg(I(G))\leq \nu(G)+1.$
\end{proof}

\begin{example}
Let $G$ be the graph in Figure \ref{fig:runningex}. Note that $\nu(G \setminus \Gamma (G))=2 < \nu(G)=3.$ Thus by Theorem \ref{unireg_2}, we have $\reg (G)=4.$
\end{example}

Our next step is to provide sufficient conditions for $\reg(I(G))=\nu(G)+2$ when $G$ is a unicyclic graph.

\begin{thm}\label{regupper} Let $G$ be a unicyclic graph with cycle $C_n$. If  $\nu(G \setminus \Gamma(G))=\nu(G)$  and $n \equiv 2~  (mod~3),$ then $\reg(I(G))=\nu(G)+2$.
\end{thm}

\begin{proof} Recall that $\nu(G \setminus \Gamma(G))=\nu(C_n)+\nu(\coprod_{i=1}^t H_i)$  for some $t\geq 0$ by Equation \ref{indequal}.
Then
$$
   \begin{array}{lcll}
    \reg(I(G \setminus \Gamma(G))&=&\reg(I(C_n))+\reg(I(\coprod_{i=1}^t H_i))-1 & \text{ (\cite[Lemma 8]{Wood})}\\
	& = & \nu(C_n)+2+\nu(\coprod_{i=1}^t H_i)) & \text{ (\cite[Theorem 7.6.28]{Jacques} and }\\
	& & & ~\text{ \cite[Theorem 2.18]{Zheng})}\\
    & = &  \nu(G \setminus \Gamma(G))+2 &\\
    & = & \nu(G)+2.
    \end{array}
$$
As $G \setminus \Gamma(G)$ is an induced subgraph of $G$, we have $\reg(I(G))\geq \nu(G)+2$ by Theorem \ref{Inequalities}. Hence $\reg(I(G)) = \nu(G)+2$ by Theorem \ref{unicyclicreg}.
\end{proof}

One of the main results in this section is an immediate corollary of Theorem \ref{unireg_2} and Theorem \ref{regupper}.

\begin{cor}\label{unireg1} Let $G$ be a unicyclic graph with cycle $C_n$. Then $\reg(I(G))=\nu(G)+2$ if and only if $n \equiv 2 ~ (mod~3) \textrm{ and }  \nu(G \setminus \Gamma(G))=\nu(G) .$
\end{cor}

\begin{proof}
It is known that $\reg(I(G))=\nu(G)+1$ or $\reg(I(G))=\nu(G)+2$ due to Theorem \ref{unicyclicreg}.
Thus the proof follows directly from  Lemma \ref{unireg_1}, Theorem \ref{unireg_2} and Theorem \ref{regupper}.
\end{proof}

\begin{remark}
Let $G$ be a unicyclic graph with cycle $C_n$. 
If $G$ satisfies the conditions from Corollary \ref{unireg1},
then $\reg(I(G))>3$.
\end{remark}

Recall that regularity of a unicyclic graph $G$ is either $\nu(G)+1$ or $\nu(G)+2.$ Thereby, taking the contrapositive of  Corollary \ref{unireg1} completes the characterization of unicyclic graphs with regularity $\nu(G)+1.$

\begin{cor}\label{unireg2} Let $G$ be a unicyclic graph with cycle $C_n$. Then $\reg(I(G))=\nu(G)+1$ if and only if $n \equiv \{0,1\}~(mod~3) \textrm{ or }  \nu(G \setminus \Gamma(G)) < \nu(G) .$
\end{cor}

Application of Corollary \ref{unireg1} and Corollary \ref{unireg2} yields yet another
positive result, namely, a partial answer to a question posed by H{\`a}, \cite[Problem 6.3]{Ha2}.
 \begin{cor}\label{reg-3} Let $G$ be a unicyclic graph with cycle $C_n$. Then $\reg(I(G))=3$ if and only if $\nu(G)=2$ and  $n \equiv \{0,1\}~(mod~3)$ or $\nu(G \setminus \Gamma(G)) < \nu(G)$.
 \end{cor}

For a graph $G$ on $n$ vertices, let $W(G)$ be the \textit{whiskered graph} on $2n$ vertices obtained by
adding a pendent vertex (an edge to a new vertex of degree 1) to every vertex of $G$.

As a consequence of Corollary \ref{unireg2}, we derive the following result in \cite{MSY}.

\begin{cor}\label{whisker_remark}\cite[Proposition 1.1]{MSY}
If $G = W(C_n)$ for $n \geq 3,$ then $\reg(I(G))=\nu(G)+1$.
\end{cor}
\begin{proof} If $n \equiv \{0,1\}~ (mod~3)$, then by Corollary \ref{unireg2}, $\reg(I(G))=\nu(G)+1$.
If  $n \equiv 2~(mod~3)$,
 then we can observe that $\nu(G \setminus \Gamma(G))=\nu(C_n)<\nu(G)$. Therefore by Corollary \ref{unireg2},
$\reg(I(G))=\nu(G)+1$.
\end{proof}

\section{Regularity bounds for colon ideals of cycles}\label{reg_path_cycle}

Let $\widetilde{G}$ denote the graph associated to edge ideal $(I(G)^{s+1}:M)^{\pol}$ 
where $M$ is a minimal monomial generator of $I(G)^s$ for some $s\geq 1.$ 
In this section, we investigate the regularity of $\widetilde{G}\cup F$ where $F$ is a forest 
attached to $G$ at some of its vertices.
In particular, we consider the cases when $G$ is a path and a cycle.
Furthermore, we obtain an upper bound on regularity of $ \widetilde{G}\cup F$ in terms of the 
induced matching number of $G \cup F.$ These bounds are interesting on their own but they will also 
be used later in the proofs of our main result.

Let $C_n$ be a cycle with vertices $x_1,\ldots, x_n$ (in order)  
and $M$ be a minimal monomial generator of $I(C_n)^s$ for $s \geq 1.$ 
In order to compute regularity of powers of cycles, authors of \cite{BHT} studied generators of 
$(I(C_n)^{s+1}:M)$ and bounded its regularity in terms of induced matching of $C_n.$  We start the section by rephrasing couple relevant results from the proof of 
Theorem 5.2 in \cite{BHT}.  Motivated by these results, we start developing the main machinery of this section.

\begin{remark}\label{obsEvenCn}
Let $J$ be the polarization of $(I(C_n)^{s+1}:M)$ and $\widetilde{C_n}$ be the graph associated to $J.$ It is known due to Theorem \ref{even_connec_equivalent} that $E(C_n) \subseteq E(\widetilde{C_n})$ and all the remaining edges of $\widetilde{C_n}$ come from even connections. In particular, $\{u,v\} \in  E(\widetilde{C_n})$ when $u$ and $v$ are even-connected with respect to $M$, and whiskers $\{x_{i_j},z_j\} \in  E(\widetilde{C_n})$ where $z_j$ is a new variable obtained by polarizing $x_{i_j}^2 $ if $x_{i_j}^2 \in  (I(C_n)^{s+1}:M).$  Note that $u \neq v$ in this setting.

Let $C_n^{\even}$ be the graph with all such even-connected edges $\{u,v\}.$  It is shown in \cite[ Theorem 5.2]{BHT} that deleting whisker does not change the regularity and
$$ \reg (\widetilde{C_n})= \reg (C_n \cup C_n^{\even}) \leq \nu(C_n)+1.$$
\end{remark}

Understanding the new edges coming from even-connections in a graph plays an essential role in computing the regularity of powers of an edge ideal (see \cite{Banerjee, JNS}). The following result considers the case when the even-connection paths in $G$ are independent from a subgraph of $G.$

\begin{lem}\label{EvenUnion}Let $G_1,$ $G_2$ be subgraphs of $G$ such that $E(G_1) \cup E(G_2) =E(G)$ and $E(G_1) \cap E(G_2) =\emptyset.$  Suppose $M$ is a minimal monomial generator of $I(G_1)^s$ for $s \geq 1$ such that none of the vertices in $G_2$ divides $M.$ Then
$$(I(G)^{s+1}:M) =(I(G_1)^{s+1}:M)+I(G_2).$$
\end{lem}

\begin{proof} The statement holds trivially when $E(G_2)= \emptyset.$ Thus we may assume that $E(G_2) \neq \emptyset.$ It is clear that $(I(G_1)^{s+1}:M) + I(G_2) \subseteq (I(G)^{s+1}: M)$ by Theorem \ref{even_connec_equivalent}.  
If $uv$ is a minimal generator of $(I(G)^{s+1}: M)$, then either $\{u,v\} \in E(G)$ or
$u$ and $v$ are even-connected with respect to $M.$ If $\{u,v\} \in E(G)$, then $uv \in (I(G_1)^{s+1}:M) + I(G_2)$ by Theorem \ref{even_connec_equivalent}.  

Let $M=e_1\cdots e_s$ where $e_1, \ldots, e_s \in E(G_1)$. Suppose $u$ and $v$ are even-connected in $G$ with respect to $M$.  Let $u=p_0p_1\cdots p_{2r}p_{2r+1}=v$
be an even-connection in $G$ such that $p_{2l+1}p_{2l+2}=e_i$ for some $1 \leq i \leq s$ and $0 \leq l \leq r-1.$ 
Note that each $p_i$ divides $M$ for $1 \leq i \leq 2r$ while none of the vertices of $G_2$ divides $M,$ thus $p_i \in V(G_1) \setminus V(G_2)$ for all $1 \le i \le 2r.$ If $u \in V(G_2) \setminus V(G_1),$ then $p_0p_1$ is an edge of $G$ but neither an edge of $G_1$ or $G_2$  which is a contradiction to the assumption that $E(G)= E(G_1) \cup E(G_2).$ It can be shown similarly that $v \in V(G_1).$ Thus $u$ and $v$ are even-connected in $G_1$ with respect to $M$ and the equality holds.

\end{proof}

The following example shows that Lemma \ref{EvenUnion} is no longer true if 
$G_2$ has a vertex which divides $M$.

\begin{example}
 Let $G$ be the graph as shown in Figure \ref{fig:runningex}. Let 
 $G_1$ and $G_2$ be the subgraphs of $G$ with $E(G_1)=\Big \{ \{x_1,x_2\},\{x_2,x_3\},
 \{x_3,x_4\},\{x_4,x_5\},\{x_1,x_5\}\Big \}$ and
 $E(G_2)=\Big \{ \{x_1,y_1\},\{x_1,y_2\},\{x_2,y_3\},\{x_2,y_6\},\{y_3,y_4\},\{y_4,y_5\},\{x_5,y_7\} \Big \}$
 respectively. Set $M=x_1x_5$. Then $y_2y_7 \in (I(G)^2:M)$ but $y_2y_7 \notin (I(G_1)^2:M)+I(G_2)$.
\end{example}

If $G_1$ is a path and $G_2$ is a forest in the statement of Lemma \ref{EvenUnion}, we can bound the regularity of $(I(G)^{s+1}:M)$ by the induced matching of $G.$

\begin{lem}\label{EvenPath} Let $P_n$ be a path on n-vertices and $F$ be a forest attached to $P_n$ on some of its vertices. Let $M$ be a minimal monomial generator of $I(P_n)^s$ for some $s \geq 1$ and $\widetilde{P_n}$ denote the associated graph to $(I(P_n)^{s+1}:M)^{\pol}.$ Suppose that none of the roots of $F$ divides $M.$ Then
$$ \reg (\widetilde{P_n} \cup F) \leq \nu(P_n \cup F) +1.$$
\end{lem}

\begin{proof}  If $E(F)= \emptyset$, we have 
$$\reg (I(P_n)^{s+1}:M) \leq \reg (I(P_n))=\nu(P_n)+1$$
 by \cite[Corollary 4.12 (2)]{JNS} and \cite[Theorem 2.18]{Zheng}. Thus the statement holds.
 
Suppose that  $E(F) \neq \emptyset$.  It follows from Lemma \ref{EvenUnion} that
$$(I(P_n)^{s+1}:M)+I(F)= (I(P_n \cup F)^{s+1}:M).$$
Thus $$ I(\widetilde{P_n} \cup F)= (I(P_n)^{s+1}:M)^{\pol}+I(F)= (I(P_n \cup F)^{s+1}:M)^{\pol}.$$

Note that $\reg (I(P_n \cup F)^{s+1}:M)^{\pol}= \reg (I(P_n \cup F)^{s+1}:M).$ Since $P_n \cup F$ is a forest, we have 
$$ \reg (\widetilde{P_n} \cup F) = \reg (I(P_n \cup F)^{s+1}:M) \leq \reg (I(P_n \cup F))=\nu(P_n \cup F)+1$$
by \cite[Corollary 4.12 (2)]{JNS} and \cite[Theorem 2.18]{Zheng}.

\end{proof}

Similarly, if $G_1$ is a cycle and $G_2$ is a forest in the statement of Lemma \ref{EvenUnion}, the regularity of $(I(G)^{s+1}:M)$ can be bounded by the induced matching of $G.$ 

\begin{lem}\label{EvenCycle1} Let $C_n$ be a cycle on the vertices $x_1, \ldots, x_n$  (in order) and $F$ be a forest attached to $C_n$ on some of its vertices such that $C_n \cup F$ is a unicyclic graph. Let $M$ be a minimal monomial generator of $I(C_n)^s$ for some $s \geq 1$ and $\widetilde{C_n}$ denote the associated graph to $(I(C_n)^{s+1}:M)^{\pol}.$ Suppose that none of the roots of $F$ divides $M.$ Then
$$ \reg (\widetilde{C_n} \cup F) \leq \nu(C_n \cup F) +1.$$
\end{lem}

\begin{proof} If $E(F)= \emptyset,$ the statement is clear by Remark \ref{obsEvenCn}. Assume that $E(F) \neq \emptyset.$
We use induction on $k:=|E(F)|$ where $k \geq 1.$ 

If $k=1,$ there must be a leaf, say $y,$ in $C_n\cup F$ with its unique neighbor, say $x.$ 
Note that $y$ is a leaf in $\widetilde{C_n} \cup F$ by Lemma \ref{EvenUnion}. Without loss of generality, we may assume that $x=x_1.$ It follows from \cite[Lemma 3.25]{biyikoglu_civan} that
\begin{eqnarray}\label{claimeq}
\reg (\widetilde{C_n} \cup F) = \max\{\reg (\widetilde{C_n}),~\reg (\widetilde{C_n} \setminus N_{\widetilde{C_n}} [x_1])+1\}.
\end{eqnarray}

It follows from Remark \ref{obsEvenCn} that $\reg  (\widetilde{C_n})\leq  \nu(C_n)+1.$ Thus $\reg  (\widetilde{C_n})\leq  \nu(C_n \cup F)+1$ as $C_n$ induced subgraph of $C_n \cup F.$

Let $G:=\widetilde{C_n} \setminus N_{\widetilde{C_n}} [x_1]$ and  $P:= C_n \setminus N_{C_n} [x_1].$ 
Note that $P$ is the path on the vertices $x_3, \ldots, x_{n-1}$ (in order). 
Let $\{g_1, \ldots, g_t\}$ be the collection of edges of $P$ that appear in $M$ and $M':=   g_1\ldots g_t.$
Consider the graph associated to $(I(P)^{t+1}:M')^{\pol}$ and denote this graph by $\widetilde{P}.$ We have the following useful inequality:
\begin{eqnarray}\label{eqn:path}
\begin{array}{lcll}
    \reg(\widetilde{P})&\leq &\reg(P)& \mbox{ (by \cite[Corollary 4.12 (2)]{JNS})}  \\
	& = & \nu(P)+1 & \mbox{ (by \cite[Theorem 2.18]{Zheng})}\\
    &\leq &\nu(C_n \cup F). & \mbox{ (by Observation \ref{indmatch})}
    \end{array}
\end{eqnarray}

\begin{claim}\label{claim1}  $G$ is an induced subgraph of $\widetilde{P},$ i.e., 
\begin{enumerate}
\item $V(G) \subseteq V(\widetilde{P})$ and
\item For $x_i,x_j \in V(G),$   $x_ix_j \in E(\widetilde{P})$ if and only if $ x_ix_j \in E(G).$
\end{enumerate}
\end{claim}

\begin{proof} It is clear that $V(G) \subseteq V(\widetilde{P}),$ thus $(1)$ holds. Suppose that $x_i, x_j \in V(G).$

First assume that $x_ix_j \in E(\widetilde{P}).$ Recall from Theorem \ref{even_connec_equivalent} that $x_ix_j \in E(P)$ or $x_i$ and $x_j$ are even-connected  in $P$ with respect to $M'.$ If $x_ix_j \in E(P),$ then $j=i+1$ as $P$ is the path on the vertices $x_3, \ldots, x_{n-2}$ and $x_ix_{i+1} \in E(C_n) \subseteq E(\widetilde{C_n}).$ Since $x_i, x_{i+1} \notin  N_{\widetilde{C_n}} [x_1],$ the edge $x_ix_{i+1}$ is preserved  in $G$ after the deletion of  $N_{\widetilde{C_n}} [x_1].$  Suppose that $x_ix_j \notin E(P).$ Then  $x_i$ and $x_j$ are even-connected  in $P$ with respect to $M'.$  This implies that $x_i$ and $x_j$ are even-connected in $C_n$ with respect to $M$ and  $x_ix_j \in E(\widetilde{C_n})$ where $x_i,x_j \notin N_{\widetilde{C_n}} [x_1].$   Thus $x_ix_j \in E(G).$

For the reverse direction, assume that $x_ix_j \in E(G).$ Then $x_ix_j \in E(C_n)$ or $x_i$ and $x_j$ are even-connected in $C_n$ with respect to $M$ whereas $x_i,x_j \notin N_{\widetilde{C_n}} [x_1].$  If $x_ix_j \in E(C_n),$ then $x_ix_j \in E(P)$ as $x_i, x_j  \notin N_{C_n} [x_1].$ If  $x_i$ and $x_j$ are even-connected in $C_n$ with respect to $M,$ then $x_2x_3$ and $x_{n-1}x_n$ can not appear on an even-connection path between $x_i$ and $x_j.$ Otherwise, $x_i$ or $x_j \in N_{\widetilde{C_n}} [x_1]$ by \cite[Observation 6.4]{Banerjee}, a contradiction.  Thus $x_i$ and $x_j$ are even-connected in $P$ with respect to $M'$ and $x_ix_j \in E(\widetilde{P}).$  Hence $(2)$ holds.
\end{proof}

Observe that $\reg(G) \leq \reg (\widetilde{P}) \leq \nu(C_n \cup F)$ by  Theorem \ref{Inequalities} and Equation (\ref{eqn:path}). Hence Equation (\ref{claimeq}) indicates that  $\reg (\widetilde{C_n} \cup F) \leq \nu(C_n \cup F) +1$ for $k=1.$

Suppose that $ k>1.$ Let $G:= \widetilde{C_n}\cup F.$ Then there exists a leaf $y$ in $C_n\cup F$  with its unique neighbor, say $x,$ and let $e:=\{x, y\} \in E(F).$  
It follows from \cite[Lemma 3.25]{biyikoglu_civan} that
\begin{eqnarray*}\label{equation08}
\reg (G) = \max\{\reg (G \setminus e),~\reg (G  \setminus N_{G} [x])+1\}.
\end{eqnarray*}

Note that $G \setminus e= \widetilde{C_n} \cup (F \setminus e)$  is an induced subgraph of $\widetilde{C_n} \cup F.$ Thus application of the induction hypothesis to $G \setminus e$ results with the following inequality.  
$$\reg(G \setminus e)\leq \nu(C_n \cup (F \setminus e))+1\leq \nu(C_n\cup F)+1.$$

Let $H:= G \setminus N_{G} [x].$ It suffices to show that $\reg (H) \leq \nu(C_n \cup F)$ to complete the proof. In order to achieve this inequality, we consider the following three cases.

\noindent
\textsc{Case 1:} Suppose $N_{G}[x] \cap V(C_n)=\emptyset$. In this case, we observe that
$$H= G \setminus N_{G} [x] =\widetilde{C_n} \cup (F \setminus N_{F} [x]).$$
Thus
$$
\begin{array}{lcll}
    \reg(H)&\leq &\nu(C_n \cup (F \setminus N_{F} [x])) +1 & \mbox{ (by the induction hypothesis)}  \\
	& = & \nu(C_n \cup F) &  \mbox{ (by Observation \ref{indmatch})}.
    \end{array}
$$

\noindent
\textsc{Case 2:} Suppose $N_{G}(x) \cap V(C_n)= \{x_i\}$ for some $1\leq i \leq n.$ 
Without loss of generality, we may assume that $x_i=x_1.$ In this case $x_1$ can not divide $M$ by our assumption as $x_1$ is a root of $F.$
This implies that $e_1=x_1x_2$ and $e_n=x_1x_n$ can not appear in $M.$

Let $P:= C_n \setminus x_1,$ namely $P$ is the path on the vertices $x_2, \ldots, x_n$ (in order). Notice that all the edges that appear in $M$ are edges in $P.$ Let  $\widetilde{P}$ be the graph  associated to $(I(P)^{s+1}:M)^{\pol}.$ 
%By \cite[Proposition 3.5]{AB}, $(I(P)^{s+1}:M)^{\pol}$ is a quadratic squarefree monomial ideal. 
Observe that if $x_i$ is even-connected to $x_1$ in $C_n$ with respect to $M,$ then $x_1x_i$ is not an edge in $H.$ It follows that $\widetilde{C_n} \setminus x_1 = \widetilde{P}$ and
\begin{eqnarray*}
H &=& (\widetilde{C_n} \cup F)   \setminus (\{x_1\} \cup N_{F}[x])\\
&=&  (\widetilde{C_n} \setminus x_1)   \cup  (F \setminus N_{F}[x])\\
& =&\widetilde{P} \cup (F  \setminus N_{F}[x]). 
\end{eqnarray*}

Therefore, we have 
$$ \begin{array}{lcll}
    \reg(H)&\leq &\nu(P  \cup (F  \setminus N_{F}[x])) +1 & \mbox{ (by Lemma \ref{EvenPath})}  \\
	& \leq & \nu(P\cup F) &  \mbox{ (by Observation \ref{indmatch})}\\
	& \leq & \nu(C_n\cup F) &  \mbox{ (since $P$ is an induced subgraph of $C_n$)}
    \end{array}
$$

\noindent
\textsc{Case 3:} Suppose $x=x_i$ for some $1\leq i \leq n.$ Without loss of generality, we may assume that $x_i=x_1.$  Let $P:= C_n \setminus N_{C_n} [x_1],$ namely $P$ is the path on the vertices $x_3, \ldots, x_{n-1}.$ Let $\{g_1, \ldots, g_t\}$ be the collection of edges of $P$ that appear in $M.$  Consider the graph associated to $(I(P)^{t+1}:M')^{\pol}$ 
where $M'=g_1\ldots g_t$ and denote this graph by $\widetilde{P}.$

Notice that 
\begin{eqnarray*}
H &=& (\widetilde{C_n} \cup F)   \setminus  N_{\widetilde{C_n}\cup F}[x_1])\\
&=&  (\widetilde{C_n} \setminus N_{\widetilde{C_n}}[x_1])   \cup  (F \setminus N_{F}[x_1]).
\end{eqnarray*}

It follows from Claim \ref{claim1} that $\widetilde{C_n} \setminus N_{\widetilde{C_n}}[x_1]$ is an induced subgraph of $\widetilde{P}.$ Thus $H$ is an induced subgraph of $\widetilde{P} \cup  (F \setminus N_{F}[x_1]).$
Therefore,
$$ \begin{array}{lcll}
  \reg(H)&\leq &  \reg(\widetilde{P}  \cup (F  \setminus N_{F}[x])) & \mbox{ (by Theorem \ref{Inequalities} )}  \\
    &\leq &\nu(P  \cup (F  \setminus N_{F}[x])) +1 & \mbox{ (by Lemma \ref{EvenPath})}  \\
	& \leq & \nu(P\cup F) &  \mbox{ (by Observation \ref{indmatch})}\\
	& \leq & \nu(C_n\cup F) &  \mbox{ (since $P$ is an induced subgraph of $C_n$)}.
    \end{array}
$$

Hence the lemma is proved.
\end{proof}

The following example shows that the equality can be achieved in Lemma \ref{EvenCycle1}.
\begin{example}
 Let $C_5 \cup F$ be the graph on $\{x_1,\ldots,x_5,y_1,\ldots,y_6\}$ as given in the figure below.
Let $M=x_3x_4$ and $\widetilde{C_5} \cup F$ be the graph associated to $(I(C_5 \cup F)^2:M)$.
 The even-connected edge is presented by the dotted line.
 
 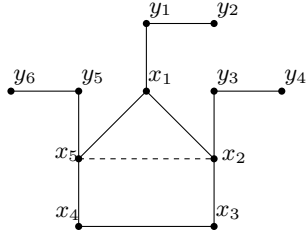
\begin{figure}[H] 
 \begin{tikzpicture}[scale=.9]
%\clip(-0.43,-0.59) rectangle (17.16,8.92);
\draw (5,4)-- (4,3);
\draw (4,2)-- (4,3);
\draw (5,4)-- (6,3);
\draw (6,3)-- (6,2);
\draw (4,2)-- (6,2);
\draw (5,4)-- (5,5);
\draw (6,5)-- (5,5);
\draw (6,3)-- (6,4);
\draw (7,4)-- (6,4);
\draw (4,4)-- (4,3);
\draw (4,4)-- (3,4);
\draw [dash pattern=on 2pt off 2pt] (4,3)-- (6,3);
\begin{scriptsize}
\fill [color=black] (5,4) circle (1.5pt);
\draw[color=black] (5.21,4.2) node {$x_1$};
\fill [color=black] (4,3) circle (1.5pt);
\draw[color=black] (3.83,3.07) node {$x_5$};
\fill [color=black] (4,2) circle (1.5pt);
\draw[color=black] (3.84,2.19) node {$x_4$};
\fill [color=black] (6,3) circle (1.5pt);
\draw[color=black] (6.3,3.04) node {$x_2$};
\fill [color=black] (6,2) circle (1.5pt);
\draw[color=black] (6.21,2.21) node {$x_3$};
\fill [color=black] (5,5) circle (1.5pt);
\draw[color=black] (5.21,5.2) node {$y_1$};
\fill [color=black] (6,5) circle (1.5pt);
\draw[color=black] (6.21,5.2) node {$y_2$};
\fill [color=black] (6,4) circle (1.5pt);
\draw[color=black] (6.21,4.2) node {$y_3$};
\fill [color=black] (7,4) circle (1.5pt);
\draw[color=black] (7.21,4.2) node {$y_4$};
\fill [color=black] (4,4) circle (1.5pt);
\draw[color=black] (4.21,4.2) node {$y_5$};
\fill [color=black] (3,4) circle (1.5pt);
\draw[color=black] (3.21,4.2) node {$y_6$};
\end{scriptsize}
\end{tikzpicture}
\caption{Graphs $C_5 \cup F$ and $\widetilde{C_5} \cup F$}
\end{figure}

It can be easily verified that $\nu(C_5 \cup F)=4.$ 
By \cite[Theorem 14]{Wood}, $\reg(I(\widetilde{C_5} \cup F))=5=\nu(C_5 \cup F)+1$.
 \end{example}

In the previous result, we focus on particular minimal monomial generators of $I(C_n)^s$ for some $s \geq 1.$ Our next result generalizes Lemma \ref{EvenCycle1} by considering any minimal monomial generator of $I(C_n)^s.$

\begin{lem}\label{EvenCycle2} Let $C_n$ be a cycle on the vertices $x_1, \ldots, x_n$  (in order) and $F$ be a forest attached to $C_n$ at some of its vertices such that $C_n \cup F$ is a unicyclic graph. Let  $E(F)=\{f_1, \ldots, f_k\}$ and $M$ be a minimal monomial generator of $I(C_n)^s$ for some $s \geq 1.$ Then
$$ \reg ((I(C_n)^{s+1}, f_1, \ldots , f_k):M) \leq \nu(C_n \cup F) +1.$$
\end{lem}

\begin{proof} Let $J:= ((I(C_n)^{s+1}, f_1, \ldots , f_k):M).$ Since all the ideals being used here are monomial ideals, we can rewrite $J$ as follows.
$$J=  (I(C_n)^{s+1}:M)+( f_1:M)+\cdots + (f_k:M).$$

Let $G$ be the graph associated to $J^{\pol}.$ Our goal is to show that $\reg (G) \leq \nu(C_n \cup F) +1.$ 

Recall that $F$ is a collection of rooted trees with roots on the cycle $C_n.$ If $M$ is a minimal generator of $I(C_n)^s$ for some $s \geq 1$ such that none of the roots divide $M,$ then the statement holds from Lemma  
\ref{EvenCycle1} since $(f_i:M)=(f_i)$ for all $1\le i \le k.$ Suppose that there exists at least one root of $F,$ say $x,$ such that $x$ divides $M.$ For the sake of simplicity we use $x$ to denote a root which is essentially a vertex $x_j$ in $C_n$ for some $1\le j \le n.$

Let $F_1$ be the collection of rooted trees in $F$ such that none of its roots divide $M$ and $F_2$ be the collection of rooted trees of $F$ such that every root in $F_2$ divides $M.$ Note that $F$ is the disjoint union of its induced subgraphs $F_1$ and $F_2.$

Observe that $(f :M)=(f)$ for all $f \in E(F_1).$ The colon ideal $(f:M)$ behaves differently when $f \in E(F_2).$  If the edge $f \in E(F_2)$ is incident to a root $x,$ then there exists a vertex $y \in V(F_2)$ such that $f=xy$ and $(f:M)=(y).$ Let $N:=\{y_1, \ldots, y_p\}$ be the collection of all such vertices $y,$ i.e.,  for any $y \in N$ there exists a root $x$ such that $xy \in E(F_2).$ 

In the light of above observations, the ideal $I(G)=J^{\pol}$ takes the following form.
\begin{eqnarray*}
I(G) &=& I(\widetilde{C_n})+ (y_1,\ldots, y_p) + I(F_1)+I(F_2 \setminus N)\\
&=&  I(\widetilde{C_n} \cup F_1)+ (y_1,\ldots, y_p) +I(F_2 \setminus N).
\end{eqnarray*}

Note that $ \{y_1, \ldots, y_p\}$ are isolated vertices of $G$ and we can drop them without effecting the regularity by \cite[Remark 2.5]{BHT}. It follows from the construction of $F_1$ and $F_2 \setminus N$ that 
\begin{eqnarray}\label{eqn:cycle2}
\nu(C_n\cup F_1) +\nu(F_2\setminus N) \leq  \nu((C_n\cup F_1) \coprod (F_2\setminus N  ) ).
\end{eqnarray}

Therefore,
$$ \begin{array}{lcll}
    \reg(G) &=& \reg (\widetilde{C_n} \cup F_1)+ \reg( F_2\setminus N )-1& \mbox{ (by \cite[Lemma 8]{Wood})}  \\
	& \leq & \nu(C_n\cup F_1) +\nu(F_2\setminus N)+1&  \mbox{ (by Lemma \ref{EvenCycle1} and \cite[Theorem 2.18]{Zheng})}\\
	& \leq & \nu((C_n\cup F_1) \coprod (F_2\setminus N  ) )+1&  \mbox{ (by Equation (\ref{eqn:cycle2}) )}\\
	& \leq & \nu(C_n\cup F ) +1 &  \mbox{ (by $(C_n \cup F_1) \coprod (F_2\setminus N) = C_n \cup (F \setminus N)$)}
    \end{array}
$$

\end{proof}

We are now ready to prove the main result of this section.

\begin{thm}\label{cycle2}
Let $C_n$ be a cycle on the vertices $x_1, \ldots, x_n$  (in order) and $F$ be a forest attached to $C_n$ at some of its vertices with  $E(F)=\{f_1, \ldots, f_k\}$ such that $C_n \cup F$ is a unicyclic graph.  Then for $s \geq 1,$
$$ \reg (I(C_n)^{s+1}, f_1, \ldots , f_k) \leq 2s+ \nu(C_n \cup F) +1.$$
\end{thm}

\begin{proof} The proof is based on induction on $s.$ We first develop a machinery to use in our induction arguments.

Suppose  $\{m_1, \ldots, m_q \}$ be the minimal monomial  generators of $I(C_n)^{s}$ for $s\geq 1$ and the monomials $\{m_1, \ldots, m_q \}$ are ordered by using the ordering given in \cite[Discussion 4.1]{Banerjee}. Let $J:=(I(C_n)^{s+1},f_1, \ldots, f_k).$ We wish to prove that $\reg (J) \leq  2s+ \nu(C_n \cup F) +1.$

Consider the following short exact sequence:
\begin{eqnarray}\label{exact01}
0 \longrightarrow \frac{R}{ (J:m_1)}(-2s) \longrightarrow \frac{R}{J} \longrightarrow \frac{R}{(J,m_1)} \longrightarrow 0.
\end{eqnarray}

Let $J_l= (J,m_1, \ldots, m_l)$ where $1\leq l \leq q$ and set $J_0=J.$ Then, for $0 \leq l \leq q-1,$ we have
\begin{eqnarray}\label{exact02}
0 \longrightarrow \frac{R}{ (J_l:m_{l+1})}(-2s) \longrightarrow \frac{R}{J_l} \longrightarrow \frac{R}{(J_{l+1})} \longrightarrow 0 .
\end{eqnarray}

Combination of Equation \ref{exact01} and Equation \ref{exact02} yields to the inequality below.
\begin{eqnarray}\label{regularity'sinequality}
\reg(J) \leq \max\{ \reg(J_l:m_{l+1})+2s,~0\leq l \leq q-1,~\reg(I(C_n)^s,f_1, \ldots, f_k)\}.
\end{eqnarray}

Understanding the ideal $J_l:m_{l+1}$ is essential to establish our upper bound. Recall that all the ideals of interest are monomial ideals. Thus it follows from \cite[Theorem 4.12]{Banerjee} that
$$(J_l:m_{l+1})= ((I(C_n)^{s+1},f_1,\dots,f_k):m_{l+1}) + (\textrm{variables}) .$$

Then we obtain the following inequality to employ in Equation (\ref{regularity'sinequality}) for all $0\leq l \leq q-1$.

$$ \begin{array}{lcll}
 \reg(J_l:m_{l+1})   &\le &  \reg((I(C_n)^{s+1},f_1,\dots,f_k):m_{l+1}) & \mbox{ (by \cite[Remark 2.5]{BHT})}  \\
	& \leq &\nu(C_n\cup F)+1 &  \mbox{ (by Lemma \ref{EvenCycle2}) )}\\
	    \end{array}
$$

and Equation (\ref{regularity'sinequality}) yields to the following. 
\begin{eqnarray}\label{eqn:final}
\reg(J) \leq \max\{ \nu(C_n\cup F)+2s+1,~\reg(I(C_n)^s,f_1, \ldots, f_k)\}.
\end{eqnarray}

Our next step is to complete the proof by using induction on $s$ with the use of above inequality. Let $s=1.$ Then Equation (\ref{eqn:final}) is
$$\reg(J) \leq \max\{ \nu(C_n\cup F)+3,~\reg(I(C_n),f_1, \ldots, f_k)\}.
$$

It follows from Theorem \ref{unicyclicreg} that $\reg (I(C_n),f_1, \ldots, f_k)\leq \nu(C_n \cup F)+2.$ Hence  $\reg(J) \leq \nu(C_n \cup F)+3$ by Equation (\ref{eqn:final}) and the statement holds for $s=1.$

Suppose $s>1.$ Then we have $ \reg(I(C_n)^s,f_1, \ldots, f_k) \leq 2s + \nu(C_n \cup F)-1$ by the induction hypothesis.  Therefore, we get the desired inequality from Equation (\ref{eqn:final}) and this completes the proof.

%$$ \reg (I(C_n)^{s+1}, f_1, \ldots , f_k) \leq 2s+ \nu(C_n \cup F) +1.$$

\end{proof}

\section{Regularity of powers of unicyclic graphs} \label{reg_power}

In this section, we obtain precise expressions for the regularity of powers of edge ideals of unicyclic graphs. We first establish an upper bound for $\reg(I(G)^s)$  in terms of $\reg I(G)$ for all $s \geq 1$ and use this bound to compute regularity explicitly. Moreover, this upper bound proves that the (below) conjecture of  Alilooee, Banerjee, Kara and H\`a holds for unicyclic graphs. We also prove that the provided upper bound is the exact value for the regularity of powers for this class of graphs.

\begin{conjecture} {[Alilooee, Banerjee, Kara, H\`a]}
Let $G$ be a finite simple graph. Then for all $s \geq 1,$

$$\reg (I(G)^s) \leq 2s+\reg (I(G)) -2.$$
\end{conjecture}

We shall use the below construction and notation for the rest of the chapter. 
Recall that a unicyclic graph  $G$ can be obtained from a cycle $C_n$ by attaching a forest to the cycle at some of its vertices. Let $F$ denote the forest attached to $C_n$ and $k:= |E(F)|.$ Note that the regularity of powers of cycles is studied in \cite{BHT}. We may assume that $k \geq 1.$
\begin{obs}\label{obs} Let $G$ be a unicyclic graph with cycle $C_n$ and a forest $F.$ We can order the edges of $F$ in such a way that deletion of the edges of $F$ with respect to that order results with an induced subgraph of $G$  at each step and that induced subgraph is also unicyclic.

Precisely, since $G$ is unicylic there exists a leaf in $G,$ say $f_1.$ Then  $G\setminus f_1$ is an induced subgraph of $G$ and a unicyclic graph. If $G\setminus f_1 \neq C_n,$ then there exists a leaf in $G \setminus f_1,$ say $f_2.$ Similarly, $(G\setminus f_1) \setminus f_2$ is  unicyclic and an induced subgraph of $G \setminus f_1$ and $G.$  Following this fashion we can order the edges of $F$ as $f_1, \ldots, f_k$ such that $f_i$ is a leaf in $G_{i-1} := G \setminus \{f_1, \ldots, f_{i-1}\}$ for $2\leq i \leq k$ and set $G_0=G,$ $G_k =C_n.$ Note that $G_i$ is unicyclic and an induced subgraph of $G_{i-1}$ and $G.$

If $f_1$ is a leaf in $G,$ we can easily observe that 
$$ I(G)^s =I(G_1)^s + \sum_{j=1}^{s}I(G_1)^{s-j} f_1^j.$$

Therefore, we get the following equalities for each $1 \leq i \leq k$
\begin{eqnarray}\label{eqleaf}
 (I(G)^s,f_1,\ldots, f_i) = (I(G_i)^s, f_1, \ldots, f_i).
 \end{eqnarray}

Note that $(I(G)^s,f_1,\ldots, f_k) =(I(C_n)^s,f_1, \ldots, f_k).$
\end{obs}
Our first result of the section introduces an upper bound for the regularity of powers for unicyclic graphs.

\begin{lem}\label{big}
If $G$ is a unicyclic graph, then for all $s \geq 1$,
$$\reg(I(G)^s) \leq 2s+\reg(I(G))-2.$$
\end{lem}

\begin{proof}
 The statement is clear for $s=1.$ Assume that $s \geq 2.$
We consider the following short exact sequence:
\begin{eqnarray*}
0 \longrightarrow \frac{R} {(I(G)^s:f_1)}(-2) \longrightarrow \frac{R}{I(G)^s} \longrightarrow \frac{R}{(I(G)^s,f_1)} \longrightarrow 0.
\end{eqnarray*}
Since $f_1$ is a leaf of $G$, by \cite[Lemma 2.10]{Morey}, $(I(G)^s:f_1)=I(G)^{s-1}$. 
By making use of Equation \ref{eqleaf}, the short exact sequence yields to  the following inequality
\begin{eqnarray*}
\reg(I(G)^s) \leq \max \{\reg(I(G)^{s-1})+2,~\reg((I(G_1)^s,f_1)) \}.
\end{eqnarray*}
We have $\reg (I(G))^{s-1} +2 \leq 2s +\reg(I(G)) -2$ by the induction hypothesis. 
Thus it remains to show that
$\reg(I(G_1)^s,f_1)) \leq 2s + \reg(I(G))-2.$  This follows from the following more general claim:
\vskip 1mm
\noindent
\textbf{Claim: }For each $ 1 \le i \le k,$ denote the induced subgraph of $G$ whose edge set is $\{f_1, \ldots, f_i\}$ by $F_i$ and $G_i = G \setminus F_i.$ Let $F_i'$ be any induced subgraph of $F_i$ such that $G_i \cup F_i'$ is an induced subgraph of $G.$ Then for all $s\ge 1,$
$$ \reg (I(G_i)^s+I(F_i')) \le 2s+\reg I(G) -2.$$
\vskip 1mm
\noindent
\textit{Proof of the claim:} We prove the claim by using induction on $s$. 
If $s=1$, the statement holds as $\reg (I(G_i \cup F_i')) \leq \reg(I(G))$ by Theorem \ref{Inequalities} (1).
 Suppose $s>1.$ Consider the following exact sequence:
 %In this case, we employ the following short exact sequence to bound  $\reg (I(G_i)^s+I(F_i')) .$
%Suppose that for each $1\le i \le k,$ we have $\reg (I(G_i)^{s-1} +I(F_i')) \le 2(s-1)+\reg I(G)-2.$ 
\begin{eqnarray*}
0 \longrightarrow \frac{R} {((I(G_i)^s+I(F_i')):f_{i+1}}(-2) \longrightarrow \frac{R}{I(G_i)^s+I(F_i')} \longrightarrow \frac{R}{I(G_i)^s+I(F_i') +(f_{i+1})} \longrightarrow 0.
\end{eqnarray*}
Recall from Observation \ref{obs} that $f_{i+1}$ is a leaf in $G_i$ for each $1\le i \le k-1.$ Thus 
$$(I(G_i)^s+I(F_i')):f_{i+1} =(I(G_i)^s:f_{i+1}) +  (I(F_i'):f_{i+1})= I(G_i)^{s-1} + I(F_i'')+(\text{variables})$$
where $F_i''$ is the graph whose edge ideal is $ I(F_i''):= (I(F_i'):f_{i+1}).$ 
Note that $ I(F_i'')$ is either $I(F_i')$ or $ I(F_i''\setminus N[f_{i+1}]).$  
It follows that $F_i''$ is an induced subgraph of $F_i$ and $G_i \cup F_i''$ is an induced subgraph of $G.$ 
Furthermore, we have
$$ I(G_i)^s +I(F_i')+ (f_{i+1})= I(G_{i+1})^s+ I(F_{i+1}'),$$
where $F_{i+1}'$ is an induced subgraph of $F_{i+1}$ with the edge set $E(F_i')\cup \{f_{i+1}\}.$  
It can be easily verified that $F_{i+1}'$ is an induced subgraph of $F_{i+1}$ 
by making use of the condition $G_i \cup F_i'$ is an induced subgraph of $G.$ 
Since $G_{i+1}\cup F_{i+1}'= G_i \cup F_i',$ it must be an induced subgraph of $G$ for $1\le i \le k-1.$
It follows that, for each $1 \leq i \le k-1,$ we have
\begin{eqnarray*}
\reg(I(G_i)^s +I(F_i')) \leq \max \{\reg(I(G_i)^{s-1}+I(F_i''))+2,~\reg(I(G_{i+1})^s+ I(F_{i+1}')) \}.
\end{eqnarray*}
Combining the above inequalities for each $1 \leq i \le k-1$  yields to the following:
\begin{eqnarray}\label{last}
\reg(I(G_i)^s +I(F_i')) \leq \max_{ i \le q \le k-1} \{\reg(I(G_q)^{s-1}+I(F_q''))+2 ,
~\reg(I(G_k)^s+ I(F_{k}')) \},
\end{eqnarray}
where $F_q''$ is an induced subgraph of $F_q$ such that $G_q \cup F_q''$ is an induced subgraph of $G.$
It follows from the induction hypothesis that
$$\reg(I(G_q)^{s-1} +I(F_q'))+2  \le 2s+\reg (I(G))-2. $$
Thus it remains to show that
$$\reg(I(G_k)^s+ I(F_{k}') \leq 2s+\reg (I(G))-2.$$

Note that $G_k=C_n$ and $G = C_n \cup F_k.$ If $C_n\cup F_k'$ is connected, then, by Theorem \ref{cycle2},
and \cite[Lemma 2.2]{Katzman}, we have
\[
 \reg (I(C_n)^{s}+ I(F_k'))  \le   2s+ \nu(C_n \cup F_k') -1 \leq 2s+ \nu(C_n \cup F_k) -1 \leq 2s +\reg(I(G))-2.
\]
%$$ \begin{array}{lcll}
% \reg (I(C_n)^{s}+ I(F_k')) & \le &  2s+ \nu(C_n \cup F_k') -1  & \mbox{ (by Theorem \ref{cycle2})} \\
% &\le &2s+ \nu(C_n \cup F_k) -1  & \mbox{ (since $C_n\cup F_k'$ is an induced subgraph of $G$)}  \\
%&\le & 2s +\reg(I(G))-2 & \mbox{ ( by \cite[Lemma 2.2]{Katzman})}.
%\end{array}$$

 If $C_n\cup F_k'$ is not connected, let $C_n\cup F_k':=(C_n \cup \dot{F}) \coprod \ddot{F}$ where $C_n \cup \dot{F}$ is connected and $\ddot{F}:= F_k' \setminus \dot{F}.$  Then,
%Suppose $C_n\cup F_k'$ is not connected. Let $\dot{F}$ be the collection of trees of $F_k'$ with roots on $C_n$ and $\ddot{F} := F_k' \setminus \dot{F}$. Since $C_n \cup \dot{F}$ is connected and $\ddot{F}$ is disjoint, we get
$$ \begin{array}{lcll}
 \reg (I(C_n)^{s}+ I(F_k')) &=&  \reg (I(C_n)^{s}+ I(\dot{F})+I(\ddot{F}))  &\\
 &=& \reg (I(C_n)^{s}+ I(\dot{F}))+\reg (I(\ddot{F}))-1  & \mbox{ (by \cite[Lemma 8]{Wood})}  \\
 &=&  \reg (I(C_n)^{s}+ I(\dot{F})) +\nu(\ddot{F}) & \mbox{ (by \cite[Theorem 2.18]{Zheng})} \\
 &\le& 2s +\nu(C_n\cup \dot{F}) -1 +\nu(\ddot{F})&  \mbox{ (by Theorem \ref{cycle2}})\\
 &=& 2s +\nu(C_n \cup F_k')-1 &  \mbox{ ( since  $F_k'= \dot{F} \ \coprod \ddot{F})$}\\
 &\le & 2s +\reg(I(G))-2. & \mbox{ (by \cite[Lemma 2.2]{Katzman}})
\end{array}$$

%If $C_n \cup \dot{F} =C_n,$ the desired inequality is obtained similarly by using \cite[Theorem 5.2]{BHT}.
%The last inequality follows from the previous set of inequalities. 

%Suppose $C_n\cup F_k'=C_n \coprod \dot{F}$, where $\dot{F}$ is forest. It follows from 
%\cite[Corollary 3.2]{Herzog2007}, \cite[Theorem 2.18]{Zheng} and \cite[Theorem 4.5 and Theorem 4.7]{BHT}, that
%\begin{align*}
% \reg(I(C_n)^s+I(\dot{F}))& \leq \reg(I(C_n)^s)+\reg(I(\dot{F}))-1 = 2s+\nu(C_n)-1+\nu(\dot{F})\\
% & \leq 2s+\nu(C_n \cup F_k)-1
% \leq 2s+\reg(I(G))-2.
%\end{align*}
%By (\ref{last}),
%$ \reg(I(G_i)^s +I(F_i')) \le 2s +\reg(I(G))-2$ for each $ 1 \le i \le k$ and the claim holds.
\end{proof}

Our main result of the paper shows that regularity is equal to the upper bound given in Lemma \ref{big}.

\begin{thm}\label{main} If $G$ is a unicyclic graph, then for all $s \geq 1$,
$$\reg(I(G)^s)=2s+\reg(I(G))-2.$$
\end{thm}
\begin{proof} If  $\reg(I(G))=\nu(G)+1$, then 
by Lemma \ref{big} and \cite[Theorem 4.5]{BHT} for all $s \geq 1$
$$\reg(I(G)^s)=2s+\nu(G)-1=2s+\reg(I(G))-2.$$
Let $H=G\setminus \Gamma(G)=C_n \coprod (\coprod_{i=1}^t H_i)$
where $n \equiv 2 (mod~3)$.
 Note that $\nu(H)=\nu(C_n)+\nu(H_1)+\ldots+\nu(H_t)$.
 \vskip 1mm
 \noindent
 \textbf{Claim: }$\reg(I(H)^s)=2s+\nu(H)$ for all $s \geq 1$.
 \vskip 1mm
 \noindent
 \textit{Proof of the claim:} It follows from \cite[Theorem 4.7, Theorem 5.2]{BHT} and \cite[Theorem 5.7]{nguyen_vu} that for $s \geq 3$ we have 
$$\reg(I(H)^s)=2s+\nu(H).$$ 

The case $s=1$ is proved by using \cite[Lemma 8]{Wood} and the remaining case  $s=2$ follows from \cite[Proposition 2.7 (ii)]{HTT}. Thus the claim is proved.

If $\reg(I(G))=\nu(G)+2$, then by Corollary \ref{unireg1}, $\nu(H)=\nu(G)$. Hence $\reg(I(H)^s)=2s+\nu(G)$.
Therefore it follows from \cite[Corollary 4.3]{BHT} and Lemma \ref{big} that for all $s \geq 1$,
$$\reg(I(G)^s)=2s+\nu(G)=2s+\reg(I(G))-2.$$
\end{proof}

\begin{remark}\label{main_remark}
 The equality given in Theorem \ref{main} is not true when $G$ is a bicyclic graph.
 For example, if $$I=(x_1x_2,x_2x_3,x_3x_4,x_4x_5,x_1x_5,x_1x_6,x_6x_7,x_6x_8,x_8x_9,x_9x_{10},
 x_{10}x_{11},x_{11}x_{12},x_{12}x_8),$$
 then computation in Macaulay2 \cite{M2} shows that the $\reg(I)=5$, $\reg(I^2)=6$, $\reg(I^3)=8$,
 $\reg(I^4)=10$ and $\reg(I^5)=12$.
\end{remark}

Since whiskered cycle graphs are unicyclic graphs, we derive the main results of  \cite{MSY}  from Corollary \ref{whisker_remark} and Theorem \ref{main}.

\begin{cor}\cite[Theorem 2.5]{MSY} Let $G = W(C_n)$ be a whiskered cycle graph. Then for all
$s \geq 1$,
$$\reg(I(G)^s)=2s+\nu(G)-1.$$
\end{cor}

\begin{remark} \label{re:dis_unicycle}
Our main focus in this paper is on regularity of powers of connected unicyclic graphs. However, one can extend the results to disconnected
unicyclic graphs and provide a precise expression for $\reg(I(G)^s)$ when $G$ is a disconnected unicyclic graph.

Suppose $G=G_1 \coprod (\coprod_{i=2}^t G_i)$ where $G_1$ is a connected unicyclic graph and $G_2,\ldots,G_t$
are trees. By Theorem \ref{main} and \cite[Theorem 4.7]{BHT}, we have
\begin{enumerate}
 \item $\reg(I(G_1)^s)=2s+\reg(I(G_1))-2$ for all $s \geq 1$.
 \item $\reg(I(\coprod_{i=2}^t G_i)^s)=2s+\nu(\coprod_{i=2}^t G_i)-1=2s+\reg(I(\coprod_{i=2}^t G_i))-2$ for all $s \geq 1$.
\end{enumerate}
By \cite[Theorem 5.7]{nguyen_vu}, we obtain $\reg(I(G)^s)=2s+\reg(I(G))-2$ for all $s \geq 2$.
\end{remark}

\vskip 2mm \noindent
\textbf{Acknowledgement:}
We would like to express our gratitude and appreciation to T\`ai Huy H\`a,  A. V. Jayanthan, and Arindam Banerjee for many useful suggestions related to this paper. Authors are deeply grateful to the referee for their useful comments and suggestions which improved the manuscript in many ways.
We heavily used commutative algebra package, Macaulay2 \cite{M2}, for verifying our results.
The third author is funded by National Board for Higher Mathematics, India.

%\nocite*{}
\bibliographystyle{abbrv}  %% or
\bibliography{ref_unicyclic}

\end{document}